\documentclass[preprint,11pt]{elsarticle}

\usepackage{amsfonts, amsmath, amscd}
\usepackage[psamsfonts]{amssymb}

\usepackage{amssymb}

\usepackage{pb-diagram}

\usepackage[all,cmtip]{xy}

\usepackage[usenames]{color}

\newtheorem{theorem}{Theorem}[section]
\newtheorem{lemma}[theorem]{Lemma}
\newtheorem{corollary}[theorem]{Corollary}

\newtheorem{remark}[theorem]{Remark}
\newtheorem{proposition}[theorem]{Proposition}
\newtheorem{definition}[theorem]{Definition}
\newtheorem{example}[theorem]{Example}

\newtheorem{claim}[theorem]{Claim}

\newtheorem{problem}[theorem]{Problem}

\newproof{proof}{Proof}

\numberwithin{equation}{section}
\numberwithin{theorem}{section}

\newcommand{\w}{\omega}

\newcommand{\NN}{\mathbb{N}}

\newcommand{\IR}{\mathbb{R}}

\newcommand{\DD}{\mathcal{D}}

\newcommand{\KK}{\mathcal{K}}

\newcommand{\AAA}{\mathcal{A}}

\newcommand{\cl}{\mathrm{cl}}
\newcommand{\Ra}{\Rightarrow}

\newcommand{\LRa}{\Leftrightarrow}

\newcommand{\CC}{C_k}

\newcommand{\SM}{{\setminus}}

\long\def\adds/#1\eadds/{\color{blue}#1\color{black}}

\def\cl#1{\overline{#1}}
\def\R{\mathbb{R}}
\def\bo{\mathbf{1}}

\begin{document}

\begin{frontmatter}

\title{$\kappa$-spaces}

\author{Saak Gabriyelyan}
\ead{saak@math.bgu.ac.il}
\address{Department of Mathematics, Ben-Gurion University of the Negev, Beer-Sheva, P.O. 653, Israel}

\author{Evgenii Reznichenko}
\ead{erezn@inbox.ru}
\address{Department of Mathematics, Lomonosov Mosow State University, Moscow, Russia}

\begin{abstract}
We say that a Tychonoff space $X$ is a $\kappa$-space if it is homeomorphic to a closed subspace of $C_p(Y)$ for some locally compact space $Y$. The class of $\kappa$-spaces is strictly between the class of Dieudonn\'{e} complete spaces and  the class of $\mu$-spaces. We show that  the class of $\kappa$-spaces has nice stability properties, that allows us to define the $\kappa$-completion $\kappa X$ of $X$ as the smallest $\kappa$-space in the Stone--\v{C}ech compactification $\beta X$ of $X$ containing $X$. For a point $z\in\beta X$, we show that (1) if $z\in\upsilon X$, then the Dirac measure $\delta_z$ at $z$  is bounded on each compact subset of $C_p(X)$, (2) $z\in \kappa X$ iff $\delta_z$ is continuous on each compact subset of $C_p(X)$ iff $\delta_z$ is continuous on each compact subset of $C_p^b(X)$, (3) $z\in\upsilon X$ iff $\delta_z$ is bounded on each compact subset of $C_p^b(X)$. It is proved that $\kappa X$ is the largest subspace $Y$ of $\beta X$ containing $X$ for which $C_p(Y)$ and $C_p(X)$ have the same compact subsets, this result essentially generalizes a known result of R.~Haydon.
\end{abstract}


\end{frontmatter}


\section{Introduction}


We denote by $C_p(X)$ the space $C(X)$ of all continuous real-valued functions on a Tychonoff space $X$ endowed with the pointwise topology. It is well-known that each Tychonoff space $X$ is homeomorphic to a closed subspace of the space  $C_p(Y)$ for some Tychonoff space $Y$ (one can take $Y=C_p(X)$). The classical characterization of realcompact spaces states that $X$ is realcompact if, and only if, $X$ is homeomorphic to a closed subspace of $C_p(D)$ for some discrete (the same sequential, see Proposition \ref{p:cl-embed-sequential} below) space $D$. These results motivate us to consider the following general problem.

\begin{problem} \label{prob:1}
Let $\mathcal{Y}$ be a family of Tychonoff spaces. Characterize those spaces $X$ which are homeomorphic to a closed subspace of $C_p(Y)$ for some $Y\in \mathcal{Y}$.
\end{problem}

In this article we solve Problem \ref{prob:1} for one of the most important classes of Tychonoff spaces, namely, the class $\mathcal{LCS}$ of all locally compact spaces. We shall say that a Tychonoff space $X$ is a {\em $\kappa$-space} if it is homeomorphic to a closed subspace of $C_p(Y)$ for some locally compact space $Y$. In Theorem \ref{t:k-Cp-spaces} we give several characterizations of $\kappa$-spaces, in which we show in particular that $X$ is a $\kappa$-space if, and only if, it is homeomorphic to a closed subspace of $C_p(Y)$ for the topological sum $Y$ of a family of compact spaces  if, and only if, $X$ is homeomorphic to a closed subspace of $C_p(Y)$ for some $k_\IR$-space $Y$. The condition of being a $k_\IR$-space cannot be replaced by the strictly weaker condition of being an Ascoli space, see Remark \ref{rem:Ascoli-non-kappa} (for numerous results concerning $k_\IR$-spaces and Ascoli spaces, see \cite{GR-Ascoli}). In Proposition \ref{p:kk-space-permanent} we show that the class of all $\kappa$-spaces consists of all Dieudonn\'{e} complete spaces and is contained in the class of all $\mu$-spaces. However, there are $\kappa$-spaces which are not Dieudonn\'{e} complete (Example \ref{exa:kappa-non-Dieudonne}) as well as there are $\mu$-spaces which are not $\kappa$-spaces (Example \ref{exa:mu-non-kappa}). It turns out that the class of all $\kappa$-spaces has nice stability properties (Proposition \ref{p:kk-space-permanent}) which allow us to define the {\em $\kappa$-completion} of $X$ as the smallest subspace of the Stone--\v{C}ech compactification $\beta X$ of $X$ which is a $\kappa$-space. We characterize the $\kappa$-completion $\kappa X$ of $X$ and the realcompactification $\upsilon X$ of $X$ as follows, where $\delta_z$ is the Dirac measure at a point $z$.

\begin{theorem} \label{t:x-vX-kX}
Let $X$ be a Tychonoff space, and let $z\in\upsilon X$.
\begin{enumerate}
\item[{\rm(i)}] $\delta_z$ is bounded on each compact subset of $C_p(X)$.
\item[{\rm(ii)}] $z\in \kappa X$ if, and only if, $\delta_z$ is continuous on each compact subset of $C_p(X)$.
\end{enumerate}
\end{theorem}

For the space $C_p^b(X)\subseteq C_p(X)$ of all {\em bounded} functions on a Tychonoff space $X$, we have the following analogue of Theorem \ref{t:x-vX-kX} (in which we consider the points of $\beta X$, but not only of $\upsilon X$).
\begin{theorem} \label{t:x-vX-kX-bounded}
Let $X$ be a Tychonoff space, and let $z\in\beta X$.
\begin{enumerate}
\item[{\rm(i)}] $z\in\upsilon X$ if, and only if, $\delta_z$ is bounded on each compact subset of $C_p^b(X)$.
\item[{\rm(ii)}] $z\in \kappa X$ if, and only if, $\delta_z$ is continuous on each compact subset of $C_p^b(X)$.
\end{enumerate}
\end{theorem}

We denote by $\mu X$ the {\em $\mu$-completion} of a Tychonoff space $X$ (i.e., $\mu X$ is the smallest subspace of $\beta X$ containing $X$  which is a $\mu$-space). A famous result of Haydon \cite[Proposition~2.9]{Haydon-2} states that the spaces $C_p(X)$ and $C_p(\mu X)$ have the same compact sets.
In the following theorem we essentially strengthen Haydon's assertion.

\begin{theorem} \label{t:compact-Cp(KX)}
Let $X$ be a subspace of a Tychonoff space $Y$ such that  $X\subseteq Y\subseteq \beta X$. Then the following assertions are equivalent:
\begin{enumerate}
\item[{\rm(i)}] $C_p(Y)$ and $C_p(X)$ have the same compact subsets;
\item[{\rm(ii)}] $C_p^b(Y)$ and $C_p^b(X)$ have the same compact subsets;
\item[{\rm(iii)}] $Y\subseteq \kappa X$.
\end{enumerate}
\end{theorem}
Consequently, $\kappa X$ can be characterized as the largest subspace $Y$ of $\beta X$ consisting of $X$ for which $C_p(Y)$ and $C_p(X)$ have the same compact subsets.

The article is organized as follows. In Section \ref{sec:preliminary} we recall basic notions used in the article and  prove some auxiliary results, see for example Proposition \ref{p:com-met-Cp} which is of independent interest. Permanent properties of $\kappa$-spaces are established in Section \ref{sec:k-Cp-spaces}. In the last Section \ref{sec:main} we prove Theorems \ref{t:x-vX-kX}, \ref{t:x-vX-kX-bounded} and \ref{t:compact-Cp(KX)} and give the aforementioned distinguished Examples \ref{exa:kappa-non-Dieudonne} and \ref{exa:mu-non-kappa}.


\section{Preliminary results} \label{sec:preliminary}



Let $X$ be a Tychonoff space. We denote by $\KK(X)$ the family of all compact subsets of $X$. Recall that a subset $A$ of $X$ is called {\em functionally bounded} if $f(A)$ is a bounded subset of $\IR$ for every $f\in C(X)$. A Tychonoff space $X$ is called a {\em $\mu$-space} if every functionally bounded subset of $X$ has compact closure.   The {\em Dieudonn\'e completion of $X$} is denoted by $\mathcal{D} X$. We shall say that $X$ is {\em closely embedded into } a Tychonoff space $Y$ if there is an embedding $f:X\to Y$ whose image $f(X)$ is a closed subspace of $Y$.

Recall that a Tychonoff space $X$ is called
\begin{enumerate}
\item[$\bullet$] {\em Fr\'{e}chet--Urysohn} (FU) if for any cluster point $a\in X$ of a subset $A\subseteq X$ there is a sequence $\{ a_n\}_{n\in\w}\subseteq A$ which converges to $a$;
\item[$\bullet$] {\em sequential} if for each non-closed subset $A\subseteq X$ there is a sequence $\{a_n\}_{n\in\NN}\subseteq A$ converging to some point $a\in \bar A\setminus A$;
\item[$\bullet$] a {\em $k$-space} if for each non-closed subset $A\subseteq X$ there is a compact subset $K\subseteq X$ such that $A\cap K$ is not closed in $K$ ;
\item[$\bullet$] a {\em $k_\IR$-space} if every $k$-continuous function $f:X\to\IR$ is continuous ($f$ is {\em $k$-continuous} if the restriction of $f$ to any $K\in\KK(X)$ is continuous);
\item[$\bullet$] {\em $\kappa$-Fr\'{e}chet--Urysohn} ($\kappa$-FU) if for every open subset $U$ of $X$ and every $x\in \overline{U}$, there exists a sequence $\{x_n\}_{n\in\w} \subseteq U$ converging to $x$;
\item[$\bullet$] an {\em Ascoli space} if every compact subset $\KK$ of $\CC(X)$ is evenly continuous (i.e., if the map $(f,x)\mapsto f(x)$ is continuous as a map from $\KK\times X$ to $\IR$).
\end{enumerate}
The notion of  $\kappa$-Fr\'{e}chet--Urysohn spaces was introduced by Arhangel'skii (and studied for example in \cite{Gabr-B1}). Ascoli spaces were defined in \cite{BG}.  In \cite{Noble} Noble proved that any $k_\IR$-space is Ascoli, however, the converse is not true in  general, see  \cite{BG}.
Recall that $X$ has {\em countable tightness} ($t(X)=\aleph_0$) if for any cluster point $a$ of a subset $A$ of $X$, there is a countable set $A_0 \subseteq A$ such that $a\in \overline{A_0}$.
The relationships between the introduced notions are described in the next diagram in which none of the implications is reversible:
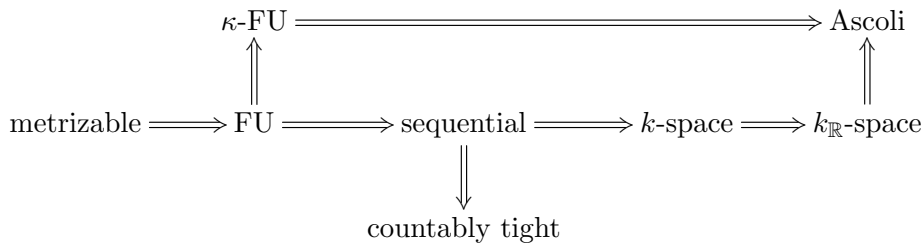
\[
\xymatrix{
& \mbox{$\kappa$-FU}  \ar@{=>}[rrr] &&& \mbox{Ascoli}\\
\mbox{metrizable} \ar@{=>}[r]  &\mbox{FU} \ar@{=>}[r]\ar@{=>}[u] & \mbox{sequential} \ar@{=>}[r]\ar@{=>}[d] & \mbox{$k$-space} \ar@{=>}[r] & \mbox{$k_\IR$-space} \ar@{=>}[u]\\
&& \mbox{countably tight} & &}
\]
Note that, by the main result of \cite{GGKZ-2}, the space $C_p(\w_1)$ is $\kappa$-Fr\'{e}chet--Urysohn  which is not a $k_\IR$-space.

Let $\mathcal{M}$ be a family of subspaces of a Tychonoff space $X$. Recall that  $X$ is called {\em strongly functionally generated} by the family $\mathcal{M}$ if for each real-valued discontinuous function $f$ on $X$, there is an $M \in\mathcal{M}$ such that the function $f{\restriction}_M : M \to \IR$ is discontinuous. Recall also that $X$ is {\em functionally generated} by the family $\mathcal{M}$ if for every discontinuous function $f: X \to \IR$, there is an $M \in\mathcal{M}$ such that the function  $f{\restriction}_M$ cannot be extended to a real-valued continuous function on the whole space $X$.

For Tychonoff spaces $X$ and $Y$ and a cardinal $\lambda$, a mapping $f:X\to Y$ is called {\em strongly $\lambda$-continuous} if for every $A\subseteq X$ of cardinality $|A|\leq \lambda$, there is a continuous function $g:X\to Y$ such that $f{\restriction}_A=g{\restriction}_A$. The {\em $R$-tightness $t_R(X)$} of $X$ is the smallest infinite cardinal $\lambda$ such that each strongly $\lambda$-continuous real-valued function on $X$ is continuous.
We shall use the following theorem of Okunev \cite{Okunev-85} (see also Exercise 438 of \cite{Tkachuk-1}).
\begin{theorem}[\cite{Okunev-85}] \label{t:Okunev}
For every Tychonoff space $X$, we have
\[
\upsilon C_p(X)=\big\{ f\in\IR^X: f \mbox{ is strictly $\w$-continuous}\big\}.
\]
\end{theorem}

If $f:X\to Y$ is a mapping between Tychonoff spaces, the {\em adjoint map} $f^{\#}:C_p(Y)\to C_p(X)$ of $f$ is defined by $f^{\#}(g)(x):=g\big(f(x)\big)$.

We shall use the following topological result.

\begin{lemma} \label{l:closed-embedding}
Let $T:X\to Y$ and $R:X\to Z$ be embeddings. Then $T{\bigtriangleup} R:X\to Y\times Z$, $T{\bigtriangleup} R(x):=\big(T(x),R(x)\big)$, is an embedding. If, moreover, $T(X)$ is a closed subset of $Y$, then $T{\bigtriangleup} R(X)$ is a closed subset of $Y\times Z$.
\end{lemma}

\begin{proof}
The mapping $T{\bigtriangleup} R$ is an embedding by Theorem 2.3.20 of \cite{Eng}. Assume that $T(X)$ is a closed subset of $Y$. To show that $T{\bigtriangleup} R(X)$ is a closed subset of $Y\times Z$, let $(y,z)\not\in T{\bigtriangleup} R(X)$.

Assume that $y\not\in T(X)$. Since $T(X)$ is closed, there is a neighborhood $U$ of $y$ such that $U\cap T(X)=\emptyset$. Then $U\times Z$ is a neighborhood of $(y,z)$ such that $(U\times Z)\cap \big(T{\bigtriangleup} R(X)\big)=\emptyset$.

Assume that $y\in T(X)$. Since $T$ is an embedding, there is a unique $x\in X$ such that $T(x)=y$. As  $(y,z)\not\in T{\bigtriangleup} R(X)$, we have $z\not= R(x)$. Choose a neighborhood $V\subseteq Z$ of $z$ and a neighborhood $W\subseteq Z$ of $R(x)$ such that $V\cap W=\emptyset$. Choose a neighborhood $Q\subseteq Y$ of $y$ such that $Q\cap T(X) \subseteq T\big( R^{-1}(W)\big)$. Then $Q\times V$ is a neighborhood of  $(y,z)$. We claim that $(Q\times V)\cap \big(T{\bigtriangleup} R(X)\big)=\emptyset$. Indeed, suppose for a contradiction that there is $(y',z')\in (Q\times V)\cap \big(T{\bigtriangleup} R(X)\big)$. Then there is $x'\in X$ such that $T(x')=y'\in Q$ and $R(x')=z'\in V$. Therefore $T(x')\in Q\cap T(X)\subseteq T\big( R^{-1}(W)\big)$. Since $T$ and $R$ are bijective, we obtain $z'=R(x')\in W$. Therefore $z'\in V\cap W=\emptyset$, which is impossible. This contradiction finishes the proof.\qed
\end{proof}

The following result is of independent interest (possibly the existence of an embedding is known but hard to find explicitly stated, so we give also a detailed proof of this fact).

\begin{proposition} \label{p:com-met-Cp}
Each complete metrizable space $(X,\rho)$ is closely embedded into $C_p(K)$ for some compact space $K$.
\end{proposition}

\begin{proof}
Consider the following subset of $C_p(X)$
\[
K:=\big\{ f\in C_p(X): \|f\|_\infty\leq 1 \mbox{ and } |f(x)-f(y)|\leq \rho(x,y) \mbox{ for all } x,y\in X\big\}.
\]
It is easy to see that $K$ is a closed subspace of $\IR^X$, which is equicontinuous and contained in $C(X)$.
Being also pointwise bounded, $K$ is a compact subset of $\IR^X$ and hence of $C_p(X)$.
It remains to prove that the mapping $\varphi:X\to C_p(K)$ defined by
\[
 \varphi(x)(f):=f(x),\quad \mbox{ where } \; x\in X \; \mbox{ and }\; f\in K,
\]
is an embedding with closed image. Clearly, $\varphi$ is continuous.

To show that $\varphi$ is injective,  fix $x_0\in X$ and  observe that the function $g(x):=\tfrac{\rho(x,x_0)}{\rho(x,x_0)+1}$ belongs to $K$ by the triangle inequality. Since $g(x)\not=g(x_0)=0$, $\varphi$ is injective.

We claim that $\varphi:X\to \varphi(X)$ is open. Indeed, since $K$ is closed in $C_p(X)$, the restriction mapping $\pi_K:C_p(C_p(X))\to C_p(K)$ is open onto its image by Proposition 0.4.1(2) of \cite{Arhangel}. Recall also (see Corollary 0.5.5 of \cite{Arhangel}) that the mapping $\psi:X\to C_p(C_p(X))$, $\psi(x)(f)=f(x)$, is an embedding. Since $\varphi=\pi_K\circ\psi$, we obtain that $\varphi:X\to \varphi(X)$ is open.

It remains to prove that $\varphi(X)$ is closed in $C_p(K)$. Let $F\in \overline{\varphi(X)}$. Take a net $\{\varphi(x_i)\}_{i\in I}$ in $\varphi(X)$ converging to $F$. Then $\{\varphi(x_i)\}_{i\in I}$ is a Cauchy net in $\varphi(X)$. Since $\varphi$ is an embedding, $\{x_i\}_{i\in I}$ is a Cauchy net in $X$ considering as a uniform space. Taking into account that $X$ is complete also as a uniform space (see Proposition 8.3.5 of \cite{Eng}), Theorem 8.3.20 of \cite{Eng} implies that $x_i\to x$ for some $x\in X$. Clearly, $\varphi(x)=F$. Thus $\varphi(X)$ is closed in $C_p(K)$.\qed
\end{proof}
\smallskip

\begin{proposition}\label{p:nrest}
Let $X$ be a normal space,  and let $M\subseteq C_p(C_p(X))$.
Then there exists $L\subseteq X$ such that $|L|\leq |M|$ and $\overline{M}$ is closely embedded in $C_p(C_p(\cl L))$.
\end{proposition}

\begin{proof}
The case of finite $M$ is trivial. We assume that $M$ is infinite.
For $Q\subseteq X$ and $h\in C_p(X)$, we put $\pi_{Q}(h):=h{\restriction}_Q$  and $C_p(Q|X):=\pi_{Q}(C_p(X))\subseteq C_p(Q)$. By Proposition~0.4.1 of  \cite{Arhangel}, the mapping $\pi_{Q}: C_p(X)\to C_p(Q|X)$ is continuous and it is open provided that $Q$ is closed in $X$. For $f\in M$, by Factorization Lemma 0.2.3 of  \cite{Arhangel}, there are a countable $L_f\subseteq X$ and a continuous function $g_f: C_p(L_f|X)\to \IR$  such that $f=g_f\circ \pi_{L_f}$. Let $L=\bigcup _{f\in M} L_f$. Clearly, $|L|\leq |M|$.

We will prove that $\overline{M}$ is closely embedded in $C_p(C_p(\overline{L}))$. Since $X$ is a normal space, then, by the Tietze--Urysohn theorem, $C_p(\cl L|X)=C_p(\cl L)$. As the mapping $\pi_{\overline{L}}: C_p(X)\to C_p(\overline{L})$ is continuous and open, the mapping $\pi_{\overline{L}}$ is a quotient mapping. Therefore, by Proposition~0.4.8 of  \cite{Arhangel},  the mapping
\[
\pi_{\overline{L}}^{\#}: C_p(C_p(\overline{L}))\to C_p(C_p(X)),\quad r\mapsto r\circ \pi_{\overline{L}},
\]
is a closed embedding of $C_p(C_p(\overline{L}))$ into $C_p(C_p(X))$. It remains to prove that $M\subseteq C_p(C_p(\overline{L}))$. Let $f\in M$. Put
\[
\tilde f: C_p(\overline{L})\to \R,\ r\mapsto g_f(r{\restriction}_{L_f}).
\]
Since $L_f\subseteq \overline{L}$, for every $r\in C_p(X)$, we have
\[
\pi_{\overline{L}}^{\#}(\tilde f)(r)= \tilde f\circ \pi_{\overline{L}}(r)= g_f\big( r\circ \pi_{\overline{L}}{\restriction}_{L_f}\big)=g_f\big(r{\restriction}_{L_f}\big)= g_f\circ \pi_{L_f}(r)=f(r).
\]
Thus $\pi_{\overline{L}}^{\#}(\tilde f)=f$.\qed
\end{proof}


\section{Permanent properties of $\kappa$-spaces} \label{sec:k-Cp-spaces}


We start this section with the following theorem.

\begin{theorem} \label{t:k-Cp-spaces}
For a Tychonoff space $X$, the following assertions are equivalent:
\begin{enumerate}
\item[{\rm(i)}] $X$ is closely embedded into $\prod_{\alpha\in \AAA}C_p(K_\alpha)$ for some family $\{K_\alpha\}_{\alpha\in\AAA}$ of compact spaces;
\item[{\rm(ii)}] $X$ is  closely embedded into $C_p(Y)$, where $Y$ is the topological sum of a family of compact spaces;
\item[{\rm(iii)}] $X$ is  closely embedded into $C_p(Y)$ for some locally compact paracompact space $Y$;
\item[{\rm(iv)}] $X$ is  closely embedded into $C_p(Y)$ for some locally compact space $Y$;
\item[{\rm(v)}] $X$ is  closely embedded into $C_p(Y)$ for some $k$-space $Y$;
\item[{\rm(vi)}] $X$ is closely embedded into $C_p(Y)$ for some $k_\IR$-space $Y$;
\item[{\rm(vii)}] the diagonal mapping
\[
\psi:=\bigtriangleup_{K\in\KK\big(C_p(X)\big)} \phi_K: X\to \prod_{K\in\KK\big(C_p(X)\big)} C_p(K)
\]
is an embedding with closed image, where $\phi_K:X\to C_p(K)$ is defined by $\phi_K(x)(s):=s(x)$ $(s\in K)$.
\end{enumerate}
\end{theorem}

\begin{proof}
The implications (i)$\Leftrightarrow$(ii)$\Ra$(iii)$\Ra$(iv)$\Ra$(v)$\Ra$(vi) and (vii)$\Ra$(i) are evident.
\smallskip

(vi)$\Ra$(i) Let $Y$ be a  $k_\IR$-space for which $X$ is  closely embedded into $C_p(Y)$. Observe that, by definition, a space is a $k_\IR$-space if, and only if, it is strongly functionally generated by the family of all its compact subsets.  It follows from Theorem 2(1) of \cite{Delgadilo2000} that $C_p(Y)$ is closely embedded in $\prod_{K\in\KK(Y)}C_p(K)$.
\smallskip

(i)$\Ra$(vii) Let $T$ be an embedding of $X$ into  $\prod_{\alpha\in \AAA}C_p(D_\alpha)$ for some family $\{D_\alpha\}_{\alpha\in\AAA}$ of compact spaces. For every $\alpha\in\AAA$, let $\pi_\alpha$ denote the $\alpha$-th coordinate projection. Then the mapping $\pi_\alpha\circ T:X\to C_p(D_\alpha)$ is continuous. Denote by $f_\alpha$ the restriction onto $D_\alpha$ of the adjoint mapping
\[
\big(\pi_\alpha\circ T\big)^{\#}: C_p\big(C_p(D_\alpha)\big) \to C_p(X).
\]
Set $Q_\alpha:=f_\alpha(D_\alpha)$. Then $Q_\alpha\in \KK\big(C_p(X)\big)$ and $f_\alpha$ is a quotient map. Therefore, by Proposition 0.4.6(2) and Corollary 0.4.8(2) of \cite{Arhangel}, the adjoint mapping $f_\alpha^{\#}: C_p(Q_\alpha)\to C_p(D_\alpha)$ is an embedding with closed image. Whence, by Proposition 2.3.2 and Corollary 2.3.4 of \cite{Eng}, the product mapping
\[
F:= \prod_{\alpha\in \AAA} f_\alpha^{\#}: \prod_{\alpha\in \AAA} C_p(Q_\alpha) \to \prod_{\alpha\in \AAA} C_p(D_\alpha)
\]
is also an embedding with closed image. Let
\[
\xi:= \bigtriangleup_{\alpha\in \AAA} \;\phi_{Q_\alpha}: X\to \prod_{\alpha\in \AAA} C_p(Q_\alpha)
\]
be the diagonal map. For every $x\in X$ and each $d\in D_\alpha$, we have
\[
f_\alpha^{\#}\circ \phi_{Q_\alpha}(x)(d)=\phi_{Q_\alpha}(x)\big( f_\alpha(d)\big)=f_\alpha(d)(x)=\pi_\alpha\circ T(x)(d).
\]
Therefore $T=F\circ \xi$. Since $T$ and $F$ are embeddings with closed image, we obtain that also $\xi$ is an embedding with closed image.

Set $I:= \KK\big(C_p(X)\big)\SM \{Q_\alpha\}_{\alpha\in\AAA}$, and let $\chi:= \bigtriangleup_{K\in I} \;\phi_K: X\to \prod_{K\in I} C_p(K)$ be the diagonal map. Then $\psi=\xi{\bigtriangleup} \chi$. Therefore, by Theorem 2.3.20 of \cite{Eng}, $\psi$ is also an embedding. Applying Lemma \ref{l:closed-embedding} we obtain that $\psi$ has closed image, as desired.\qed
\end{proof}

Theorem \ref{t:k-Cp-spaces} motivates to introduce the following class of Tychonoff spaces.

\begin{definition} \label{def:kk-space} {\em
A Tychonoff spaces is called a {\em $\kappa$-space} if it satisfies the conditions (i)-(vii) of Theorem \ref{t:k-Cp-spaces}.} 
\end{definition}

The following proposition stands some permanent properties of $\kappa$-spaces.

\begin{proposition} \label{p:kk-space-permanent}
\begin{enumerate}
\item[{\rm(i)}] An arbitrary product of $\kappa$-spaces is a $\kappa$-space.
\item[{\rm(ii)}] A closed subspace of a  $\kappa$-space is a $\kappa$-space.
\item[{\rm(iii)}] Each realcompact space is a $\kappa$-space.
\item[{\rm(iv)}] If $\mathcal{M}$ is a family of $\kappa$-subspaces of a Tychonoff space $X$, then $\bigcap\mathcal{M}$ is a $\kappa$-space.
\item[{\rm(v)}] If $f:X\to Y$ is a perfect mapping and $Y$ is a $\kappa$-space, then also $X$ is a $\kappa$-space.
\item[{\rm(vi)}] If $C_p(X)$ is a $k_\IR$-space, then $X$ is a $\kappa$-space.
\item[{\rm(vii)}] Each $\kappa$-space is a $\mu$-space.
\item[{\rm(viii)}] Each metrizable space is a $\kappa$-space.
\item[{\rm(ix)}] Each Dieudonn\'{e} complete space is a $\kappa$-space.
\item[{\rm(x)}] A pseudocompact space is a $\kappa$-space if, and only if, it is compact. Consequently, the first countable pseudocompact space $\w_1$ is not a   $\kappa$-space.
\item[{\rm(xi)}] If $X$ is a $k_\IR$-space, then $C_p(X)$ is a $\kappa$-space.
\end{enumerate}
\end{proposition}

\begin{proof}
(i) follows from (i) of Theorem \ref{t:k-Cp-spaces}, and (ii) is evident.
\smallskip

(iii) Since $\IR=C_p(Y)$, where $Y$ is a singleton, $\IR$ is a $\kappa$-space. Now the assertion follows from Theorem 3.11.3 of \cite{Eng} and the clauses (i) and (ii).
\smallskip

(iv) Let $\mathcal{M}=\{Y_i\}_{i\in I}$ be a family of $\kappa$-subspaces of $X$. Then the diagonal
\[
\bigtriangleup := \big\{ (x_i)_i\in X^I: x_i=x_j \mbox{ for } i\in I\big\}
\]
is closed in $X^I$. Observe that $\bigcap \mathcal{M}$ is homeomorphic to the intersection $Y:= {\bigtriangleup} \cap \prod_{i\in I} Y_i$. Since $Y$ is a closed subspace of $\prod_{i\in I} Y_i$, (i) and (ii) imply that $\bigcap \mathcal{M}\cong Y$ is a $\kappa$-space.
\smallskip

(v) Let $i:X\to \beta X$ be an embedding. Then the diagonal mapping $f{\bigtriangleup}i:X\to Y\times\beta X$ is an embedding (Theorem~2.3.20 of \cite{Eng}) and a perfect mapping (Theorem~3.7.9 of \cite{Eng}). It follows that $X$ is homeomorphic to a closed subspace of $Y\times\beta X$.
By (iii), $\beta X$ is a $\kappa$-space. Therefore, by (i) and (ii), $X$ is a $\kappa$-space.
\smallskip

(vi) follows from the fact that $X$ is closely embedded into $C_p(C_p(X))$.
\smallskip

(vii) Let $X$ be closely embedded into $C_p(Y)$ for some $k_\IR$-space $Y$. Since each $k_\IR$-space is strongly functionally generated by its family of all compact subsets, Theorem III.4.15 of \cite{Arhangel} implies that $C_p(Y)$ is a $\mu$-space. Being a closed subspace of the $\mu$-space $C_p(Y)$,  $X$ is also a $\mu$-space.
\smallskip

(viii) Let $X$ be a metrizable space, and let $Y$ be a completion of $X$. Note that for every $y\in Y$, the space $Y\SM \{y\}$ admits a complete metric. Therefore, by  Proposition \ref{p:com-met-Cp}, $Y\SM \{y\}$ is  a $\kappa$-space. Since $X=\bigcap_{y\in Y\SM X} Y\SM \{y\}$, (iv) implies that $X$ is a $\kappa$-space.
\smallskip

(ix) Recall that Dieudonn\'{e} spaces can be characterized as closed subspaces of products of metrizable spaces, see \cite[8.5.13(a)]{Eng}. Now the assertion follows from (i), (ii) and (viii).
\smallskip

(x) If $X$ is a pseudocompact  $\kappa$-space, then, by (vii), we have $X=\overline{X}$ is compact. Conversely, any compact space is a   $\kappa$-space by (iii).

(xi) Since $C_p(X)$ is trivially closely embedded into $C_p(X)$, it is a $\kappa$-space. \qed
\end{proof}

\begin{remark} \label{rem:Ascoli-non-kappa}{\em
One can naturally ask whether in  Theorem \ref{t:k-Cp-spaces} we can add ``$X$ is a closely embedded into $C_p(Y)$ for some Ascoli space $Y$''. The answer to this question is negative. Indeed, consider the space $X=\w_1$. By (x) of Proposition \ref{p:kk-space-permanent}, the space  $\w_1$ is not a $\kappa$-space. 
By the main result of \cite{GGKZ-2}, the space $C_p(\w_1)$ is an Ascoli space which is not a $k_\IR$-space, and $\w_1$ is closely embedded into $C_p(C_p(\w_1))$. Observe that, by \cite{GGKZ},  $C_p(\w_1)$ is a $\kappa$-Fr\'{e}chet--Urysohn space. Consequently, we cannot add to Theorem \ref{t:k-Cp-spaces} even the much stronger condition on $Y$ to be a $\kappa$-Fr\'{e}chet--Urysohn space.  Note also that the space $\w_1$ shows that the condition of being metrizable in (viii) of Proposition \ref{p:kk-space-permanent} cannot be weaken to the property of being a first countable locally compact  space.}
\end{remark}

The next characterization of realcompact spaces nicely complements (in the spirit of Problem \ref{prob:1} and Theorem \ref{t:k-Cp-spaces}) the well-known characterization of realcompact spaces mentioned in Introduction.
\begin{proposition} \label{p:cl-embed-sequential}
For a Tychonoff space $X$, the following assertions are equivalent:
\begin{enumerate}
\item[{\rm(i)}] $X$ is a realcompact space;
\item[{\rm(ii)}] $X$ is  closely embedded into $C_p(Y)$ for some discrete space;
\item[{\rm(iii)}] $X$ is  closely embedded into $C_p(Y)$ for some metrizable space;
\item[{\rm(iv)}] $X$ is  closely embedded into $C_p(Y)$ for some Fr\'{e}chet--Urysohn space $Y$;
\item[{\rm(v)}] $X$ is  closely embedded into $C_p(Y)$ for some sequential space $Y$;
\item[{\rm(vi)}] $X$ is  closely embedded into $C_p(Y)$ for some space $Y$ with countable tightness.
\end{enumerate}
\end{proposition}

\begin{proof}
The equivalence (i)$\LRa$(ii) is well-known, see Theorem 3.11.3 of \cite{Eng}, and the implications (ii)$\Ra$(iii)$\Ra$(iv)$\Ra$(v)$\Ra$(vi) are trivial.
\smallskip

(vi)$\Ra$(i) By Theorem II.4.7 of \cite{Arhangel}, $t_R(Y)=\aleph_0$. Hence, by Corollary II.4.17 of \cite{Arhangel}, $C_p(Y)$ is realcompact. Being a closed subspace of $C_p(Y)$, the space $X$ is also realcompact.\qed
\end{proof}

In some important case we can reverse (iii) of Proposition \ref{p:kk-space-permanent}.

\begin{theorem} \label{t:kappa-space-realcompact}
Let $X$ be a Tychonoff space such that each compact subset of $C_p(X)$ has countable tightness. Then $X$ is a $\kappa$-space if, and only if, it is realcompact.
\end{theorem}

\begin{proof}
The sufficiency follows from (iii) of Proposition \ref{p:kk-space-permanent}. To prove the necessity, assume that  $X$ is a $\kappa$-space. Then, by Theorem \ref{t:k-Cp-spaces}, the diagonal mapping
\[
\psi:=\bigtriangleup_{K\in\KK\big(C_p(X)\big)} \phi_K: X\to \prod_{K\in\KK\big(C_p(X)\big)} C_p(K)
\]
is an embedding with closed image, where $\phi_K:X\to C_p(K)$ is defined by $\phi_K(x)(s):=s(x)$ $(s\in K)$.

Let $K$ be a compact subset of $C_p(X)$. By assumption $K$ has countable tightness.
It follows from Propositions II.4.2 and II.4.6 of \cite{Arhangel} that any space with countable tightness has countable $R$-tightness.
Hence, by Corollary II.4.17 of \cite{Arhangel},
$C_p(K)$ is realcompact. Since $\psi$ is an embedding with closed image, Theorems 3.11.4 and 3.11.5 of \cite{Eng} imply that $X$ is realcompact.\qed
\end{proof}

\begin{corollary} \label{c:kappa-space-realcompact-1}
If $X$ is a separable $\kappa$-space, then $X$ is realcompact.
\end{corollary}

\begin{proof}
Let $D$ be a countable dense subspace of $X$. Then the restriction map $R:C_p(X)\to C_p(D)$ is injective. Therefore, for every $K\in\KK\big(C_p(X)\big)$, $R{\restriction}_K$ is a homeomorphism. 
Since $C_p(D)$ is metrizable, it follows that each compact subset $K$ of $C_p(X)$ is metrizable, and hence $t(K)=\aleph_0$. Now Theorem \ref{t:kappa-space-realcompact} applies.\qed
\end{proof}

\begin{corollary} \label{c:kappa-space-realcompact-2}
Assume that a $\kappa$-space $X$ contains a dense subspace $Y$ such that $Y^n$ is Lindel\"{o}f for every $n\geq 1$. Then $X$ is realcompact.
\end{corollary}

\begin{proof}
By Theorem II.1.1 of \cite{Arhangel}, $C_p(Y)$ has countable tightness. Since the restriction map $R:C_p(X)\to C_p(Y)$ is injective, it follows that each compact subset $K$ of $C_p(X)$ has countable tightness. Thus, by Theorem \ref{t:kappa-space-realcompact}, the space $X$ is realcompact.\qed
\end{proof}

\begin{corollary} \label{c:kappa-space-realcompact-3}
Under PFA, if a $\kappa$-space $X$ contains a dense Lindel\"{o}f subspace $Y$, then $X$ is realcompact.
\end{corollary}
\begin{proof}
By Theorem IV.11.14 of \cite{Arhangel}, any compact subspace of $C_p(Y)$ has countable tightness. Since the restriction map $R:C_p(X)\to C_p(Y)$ is injective, it follows that each compact subset $K$ of $C_p(X)$ has countable tightness. Thus, by Theorem \ref{t:kappa-space-realcompact}, the space $X$ is realcompact.\qed
\end{proof}

\begin{problem} \label{prob:kappa-space-realcompact-3}
Is it possible to prove in $\mathrm{ZFC}$ that if a $\kappa$-space $X$ contains a dense Lindel\"{o}f subspace $Y$, then $X$ is realcompact?
\end{problem}

Note that Problem \ref{prob:kappa-space-realcompact-3} is equivalent to the following question.

\begin{problem} \label{prob:kappa-space-realcompact-3+}
Is it possible to prove in $\mathrm{ZFC}$ that if $X$  is a compact space and $Y\subseteq C_p(X)$ is Lindel\"{o}f, then $\cl Y$ is realcompact?
\end{problem}

\begin{example} \label{exa:Cp-kappa-non-realcompact}
Let $X=\w_1+1$. Then $C_p(X)$ is a Fr\'{e}chet--Urysohn $\kappa$-space, which is not a realcompact space.
\end{example}

\begin{proof}
Since $X$ is compact,  (xi) of Proposition \ref{p:kk-space-permanent} implies that the space $Y=C_p(X)$ is a $\kappa
$-space. By Corollary II.4.17 of \cite{Arhangel}, $C_p(Z)$ is realcompact if, and only if, the $R$-tightness $t_R(Z)$ of $Z$ is $\aleph_0$. To prove that $C_p(X)$ is not realcompact it remains to show that $t_R(X)\not=\aleph_0$. To this end, consider the function $f:X\to \IR$ defined by $f(\lambda):=0$ for $\lambda<\omega_1$ and $f(\omega_1)=1$. It is easy to see that $f$ is strictly $\w$-continuous but not continuous. Thus $t_R(X)\not=\aleph_0$.
Since $X$ is a scattered compact space, Theorem II.7.16 of  \cite{Arhangel} implies that $C_p(X)$ is a Fr\'{e}chet--Urysohn space.\qed
\end{proof}

\begin{example} \label{exa:Cp-mu-non-kappa}
Let
\[
S=\big\{ f\in [0,1]^{\w_2}: |\{\alpha<\w_2: f(\alpha)\neq 0\}|\leq\w\big\}
\]
be a $\Sigma$-product in $[0,1]^{\w_2}$, $\mathbf{1}$ be the unit constant function on $\w_2$,  and let $X=S\cup\{\mathbf{1}\}$. Then:
\begin{enumerate}
\item[{\rm(i)}] $X$ is a normal sequentially compact space,
\item[{\rm(ii)}] $C_p(X)$ is a $\mu$-space, which is not a $\kappa$-space.
\end{enumerate}
\end{example}

\begin{proof}
Since any countable subset of $S$ sits in some metrizable compact space $[0,1]^\lambda\times \{\mathbf{0}_{\w_2\setminus \lambda}\}\subseteq S$, where $\mathbf{0}_{\w_2\setminus \lambda}$ is the zero function on $\w_2\setminus \lambda$  for some countable set $\lambda\subseteq \w_2$, $S$ is a sequentially compact space. Therefore $X$ is sequentially compact as well. Thus, by the Asanov--Velichko Theorem \cite[III.4.1]{Arhangel}, $C_p(X)$ is a $\mu$-space.

\begin{claim}
$X$ is a normal space.
\end{claim}

\begin{proof}
By Theorem 1 of  \cite{Corson1959}, the $\Sigma$-product $S$ of compact metrizable spaces is a normal space.
Applying Theorem 2 of  \cite{Corson1959} we obtain that $S$ is $C^\ast$-embedded in $\beta S=[0,1]^{\w_2}$ (i.e., any bounded continuous function on $S$ is extended to a continuous function on $\beta S$).

To prove that also $X$ is a normal space, let $F$ and $G$ be disjoint nonempty closed subsets of $X$. We construct a continuous function $f: X\to \IR$ such that $f(F)=\{0\}$ and $f(G)=\{1\}$.
\smallskip

{\em Case 1: Assume that $\mathbf{1}\notin F\cup G$.} Then $F,G\subseteq S$. Since $S$ is a normal space, there exists a continuous function $g: S\to \IR$ such that $g(F)=\{0\}$ and $g(G)=\{1\}$. Since $S$ is $C^\ast$-embedded in $[0,1]^{\w_2}$, the function $g$ is extended to a continuous function $\hat g: [0,1]^{\w_2}\to \IR$. It suffice to define  $f:=\hat g{\restriction}_X$.
\smallskip

{\em Case 2: Let $\mathbf{1}\in F\cup G$.} We assume that $\mathbf{1}\in F$. If $F$ or $G$ is compact, the existence of a desired function $f$ is well-known. Assume that neither $F$ nor $G$ is compact. Let $Q=F\cap S=F\setminus\{\mathbf{1}\}$.
Since $S$ is a normal space, there exists a continuous function $g: S\to \R$ such that $g(Q)=\{0\}$ and $g(G)=\{1\}$. Since $S$ is $C^\ast$-embedded in $[0,1]^{\w_2}$, the function $g$ is extended to a continuous function $\hat g: [0,1]^{\w_2}\to \IR$. Consider two subcases.
\smallskip

{\em Subcase 2.1: Assume that $\mathbf{1}\in \overline{Q}$.} Then $\hat g(\mathbf{1})=0$. Let $f=\hat g{\restriction}_X$.
\smallskip

{\em Subcase 2.2: Assume that $\mathbf{1}\notin \overline{Q}$.} Let $h: [0,1]^{\w_2}\to \IR$ be a continuous function such that $h(\bo)=0$ and $h(Q\cup G)=\{1\}$. Let $f=h\hat g{\restriction}_X$. \qed
\end{proof}

It remains to prove that $C_p(X)$ is not a $\kappa$-space. First we prove the following two claims.

\begin{claim} \label{cl:Cp-mu-non-kappa-1}
Any compact subset $K$ of $C_p(C_p(X))$ has countable tightness.
\end{claim}

\begin{proof}
Suppose for a contradiction that  the tightness of some compact $K\subseteq C_p(C_p(X))$ is uncountable. Then, by Theorem 7.12 of  \cite{Hodel1984}, there exists a free sequence $\{f_\alpha:\alpha<\w_1\}\subseteq K$ of length $\w_1$. Recall that a sequence $\{f_\alpha:\alpha<\w_1\}$ is called a {\em free sequence} if
\[
\overline{\{f_\alpha:\alpha<\beta\}}\cap \overline{\{f_\alpha:\beta\leq \alpha<\w_1\}}=\emptyset
\]
for any $\beta<\w_1$. Let $M:=\{f_\alpha:\alpha<\w_1\}$ and $T:=\overline{M}$. Since $T$ contains a free sequence of length $\w_1$, Theorem 7.12 of  \cite{Hodel1984} implies that $T$ is a compact space of uncountable tightness. Since $X$ is a normal space, Proposition \ref{p:nrest} implies that there exists $L\subseteq X$, $|L|\leq\w_1$, such that $T$ is embedded into $C_p(C_p(P))$, where $P=\overline{L}$.

We show that $P$ is closely embedded into the $\Sigma$-product $S$. If $\mathbf{1}\notin P$, then $P$ is a closed subset of $S$ and we are done. Assume that  $\mathbf{1}\in P$. Let $L'=L\cap S$.  Since $L'\subseteq S$ and  $|L'|\leq\w_1$, then $|I|\leq\w_1$, where
\[
I=\big\{\beta<\w_2: f(\beta)\neq 0 \, \mbox{ for some }\, f\in L'\big\}.
\]
Let $\gamma\in \w_2\setminus I$. Then $f(\gamma)=0$ for all $f\in L'$. Let $U=\{g\in X: g(\gamma)>\tfrac{1}{2}\}$. Then $U$ is an open neighbourhood of $\mathbf{1}$ and $U\cap P=\{\mathbf{1}\}$.
Hence,  $\mathbf{1}$ is an isolated point of $P$. Let $g\in S\setminus P$. Then $P'= (P\cap S)\cup \{g\}$ is closed in $S$ and homeomorphic to $P$.

By (a) and (c) of Theorem 2.1 of \cite{Tkachuk2005}, $P$ is a Sokolov space. By Theorem 2.1 of \cite{Nob}, $S$ is Fr\'{e}chet--Urysohn. Since for any natural number $n$, $S$ is homeomorphic to $S^n$ we obtain that $P^n$ is also a Fr\'{e}chet--Urysohn space. By (e) of Theorem 2.1 of \cite{Tkachuk2005}, $C_p(C_p(C_p(P)))$ is a Lindel\"{o}f space. Applying Asanov's theorem \cite[I.4.1]{Arhangel}, we obtain that the space $C_p(C_p(P))$ has countable tightness.
Thus also the compact subspace $T$ of $C_p(C_p(P))$ has countable tightness, a contradiction.\qed
\end{proof}

\begin{claim} \label{cl:Cp-mu-non-kappa-2}
$C_p(X)$ is not realcompact.
\end{claim}
\begin{proof}
By Corollary II.4.17 of \cite{Arhangel}, $C_p(Z)$ is realcompact if, and only if, the $R$-tightness $t_R(Z)$ of $Z$ is $\aleph_0$. To prove that $C_p(X)$ is not realcompact it remains to show that $t_R(X)\not=\aleph_0$. To this end, consider the function $f:X\to \IR$ defined by $f(s):=0$ for $s\in S$ and $f(\mathbf{1})=1$. It is easy to see that $f$ is strictly $\w$-continuous but not continuous. Thus $t_R(X)\not=\aleph_0$.\qed
\end{proof}

By Claim \ref{cl:Cp-mu-non-kappa-1} and Theorem \ref{t:kappa-space-realcompact}, the space $C_p(X)$ is not a $\kappa$-space if, and only if, $C_p(X)$ is not realcompact. Thus, by  Claim \ref{cl:Cp-mu-non-kappa-2}, $C_p(X)$ is not a $\kappa$-space.\qed
\end{proof}


\section{Proofs of Theorem \ref{t:x-vX-kX}, Theorem \ref{t:x-vX-kX-bounded} and Theorem \ref{t:compact-Cp(KX)}} \label{sec:main}


The items (iii) and (iv)  of Proposition \ref{p:kk-space-permanent} allow us to introduce  the $\kappa$-completion of a Tychonoff space $X$ as follows.

\begin{definition} \label{def:kk-closure} {\em
For a Tychonoff space $X$, the intersection of all  $\kappa$-subspaces of $\beta X$ containing $X$ is called the {\em  $\kappa$-completion} of $X$ and is denoted by $\kappa X$.}
\end{definition}

\begin{remark} {\em
(i) By definition and (iv)  of Proposition \ref{p:kk-space-permanent}, a Tychonoff space $X$ is a $\kappa$-space if, and only if, $X=\kappa X$.

(ii) If $X$ is pseudocompact, then $\kappa X=\beta X$. Indeed, since $X$ is dense in $\kappa X$, $\kappa X$ is also pseudocompact. Hence, by (x) of Proposition \ref{p:kk-space-permanent}, $\kappa X$ is compact. As $\kappa X$ is dense in $\beta X$, it follows that $\kappa X=\beta X$.

(iii) If $X$ is not pseudocompact, then also $\kappa X$ is not pseudocompact. Indeed, if $f$ is an unbounded function on $X$, then it can be extended to an unbounded function $\bar f$ on $\upsilon X$. Since $\kappa X\subseteq \upsilon X$ by (iii) of Proposition \ref{p:kk-space-permanent}, $\bar f{\restriction}_{\kappa X}$ is also unbounded.

(iv) Let $X$ be a countably compact subspace of $\beta\NN$ such that $X\times X$ is not pseudocompact (see, for example, Example 9.15 of \cite{GiJ}). By (ii), we have $\kappa X\times \kappa X=\beta\NN\times\beta\NN$ is a compact space. On the other hand, since $X\times X$ is not pseudocompact, it follows from (iii) that $\kappa(X\times X)$ is also not pseudocompact. In particular, $\kappa(X\times X)$ and $\kappa X\times \kappa X$ are not homeomorphic.}
\end{remark}

The following inclusions immediately follows from Definition \ref{def:kk-closure} and Proposition \ref{p:kk-space-permanent}:
\begin{equation}\label{equ:inclusion}
X \subseteq \mu X \subseteq \kappa X \subseteq \DD X \subseteq \upsilon X \subseteq \beta X.
\end{equation}

Let $f:X\to Y$ be a continuous mapping between Tychonoff spaces. By Corollary 3.6.6 of \cite{Eng}, there is a unique continuous extension $\beta(f): \beta X\to \beta Y$. To describe points from $\kappa X$ given in Proposition \ref{p:description-point-kk} below, we need the next assertion.
\begin{proposition} \label{p:f-extension-kk}
Let $X$ and $Y$ be Tychonoff spaces, and let $f:X\to Y$ be a continuous mapping. Then the continuous mapping
\[
\kappa(f):=\beta(f){\restriction}_{\kappa X} : \kappa X\to \beta Y
\]
has range in $\kappa Y$, i.e. $\kappa(f)\big(\kappa X\big)\subseteq \kappa Y$. In particular, if $Y$ is a $\kappa$-space, then $\kappa(f)\big(\kappa X\big)\subseteq Y$.
\end{proposition}

\begin{proof}
Let $Z:= \beta(f)^{-1}(\kappa Y)$. Since $\beta(f)$ is a perfect mapping, it follows from Proposition 3.7.4 of \cite{Eng} that the restriction mapping $g:=\beta(f){\restriction}_Z: Z\to \kappa Y$ is also perfect. Therefore, by (v) of Proposition \ref{p:kk-space-permanent}, $Z$ is a $\kappa$-space, too. By the definition of the $\kappa$-completion we have $\kappa X\subseteq Z$, and hence $\kappa(f)\big(\kappa X\big)\subseteq \kappa Y$.\qed
\end{proof}

Since, by (xi) of Proposition \ref{p:kk-space-permanent}, $C_p(K)$ is a $\kappa$-space for every compact $K$, Proposition \ref{p:f-extension-kk} implies the following statement.

\begin{corollary} \label{c:f-extension-kk}
Let $X$ be a Tychonoff space, and let $K$ be a compact space. Then each continuous function $f:X\to C_p(K)$ can be extended to a continuous function $\hat f:\kappa X\to C_p(K)$.
\end{corollary}

Recall that if $z\in \upsilon X$, then $\delta_z(f):=\bar f(z)$, where $\bar f$ is the unique extension of $f$ onto $\upsilon X$.
We need the following assertion.
\begin{proposition} \label{p:vX-bounded-compact}
Let $X$ be a Tychonoff space, and $z\in \upsilon X$. Then $\delta_z$ is bounded on each compact subset of $C_p(X)$.
\end{proposition}

\begin{proof}
Assume that $z\in \upsilon X$. We have to show that the Dirac measure $\delta_z$ is bounded on each compact subsets of $C_p(X)$. Let $i:X\to C_p\big(C_p(X)\big)$ be the canonical embedding. Then, by Theorem 3.11.16 of \cite{Eng}, $i$ can be extended to a continuous map $\tilde i:\upsilon X \to \upsilon C_p\big(C_p(X)\big)$. Since $i$ is an embedding and the closure $Y$ of $i(X)$ in $\upsilon C_p\big(C_p(X)\big)$ is realcompact, $Y$ can be considered as a realcompact space containing $X$ as a dense subspace. Therefore, by the definition of $\upsilon X$,  $\tilde i$ is an embedding.  Since  $\tilde i$ is injective, we apply Okunev's Theorem \ref{t:Okunev} to the space $Z=C_p(X)$ to get that $\delta_z$ is strictly $\w$-continuous on $C_p(X)$. It remains to note that any strictly $\w$-continuous function is bounded on compact sets.\qed
\end{proof}

Let $X$ be a Tychonoff space, and let $n\in\w$. Define a mapping $Q_n:C_p(X)\to C_p^b(X)$ by $Q_n(f):=f\wedge n\vee(-n)$. Then $Q_n$ is continuous. Recall also that a subset $F$ of $C(X)$ is {\em uniformly bounded} if there is $C>0$ such that $|f(x)|\leq C$ for all $x\in X$ and $f\in F$. It is clear that for every $n\in\w$ and each $A\subseteq C(X)$, the set $Q_n(A)$ is uniformly bounded.
Now we describe the points of $\kappa X$.

\begin{proposition} \label{p:description-point-kk}
Let $X$ be a Tychonoff space, and let $z\in\beta X$. Then the following assertions are equivalent:
\begin{enumerate}
\item[{\rm(i)}] $z\in \kappa X$;
\item[{\rm(ii)}] for each compact space $K$, any continuous mapping $\phi:X\to C_p(K)$ can be extended to a continuous mapping $\hat\phi: X\cup\{z\} \to C_p(K)$;
\item[{\rm(iii)}] $z\in \upsilon X$ and  $\delta_z$ is $k$-continuous on $C_p(X)$;
\item[{\rm(iv)}] $z\in \upsilon X$ and  $\delta_z$ is $k$-continuous on $C_p^b(X)$;
\item[{\rm(v)}] $z\in \upsilon X$ and $\delta_z$ is $k$-continuous on every uniformly bounded compact subset of $C_p(X)$.
\end{enumerate}
\end{proposition}

\begin{proof}
(i)$\Ra$(ii) follows from Corollary \ref{c:f-extension-kk}. 
\smallskip

(ii)$\Ra$(i) Since $\kappa X$ is a $\kappa$-space, Theorem \ref{t:k-Cp-spaces} implies that for some family $\{K_\alpha\}_{\alpha\in\AAA}$ of compact spaces, there is an embedding
\[
T:\kappa X\to P:=\prod_{\alpha\in \AAA}C_p(K_\alpha)
\]
with closed image. For simplicity of notations we set $Y:=X\cup\{z\}$.  For every $\alpha\in\AAA$, let
\[
i_\alpha:\pi_\alpha\circ T{\restriction}_X, \;\; \mbox{ where }\; \pi_\alpha:P\to C_p(K_\alpha) \;\mbox{ is the $\alpha$-th projection}.
\]
By (ii), $i_\alpha$ is extended to a continuous mapping $\hat i_\alpha:Y\to C_p(K_\alpha)$. Consider the continuous mapping $\hat i:Y\to P$ defined by
\[
\hat i={\bigtriangleup}_{\alpha\in \AAA} \,\hat i_\alpha: Y\to P, \quad \hat i(y):=\big( \hat i_\alpha(y)\big)_\alpha.
\]
Let $\kappa: \kappa X\to \beta X$ and $R:Y\to \beta X$ be embeddings. Then also $T{\bigtriangleup} \kappa:\kappa X\to P\times \beta X$,  $T{\bigtriangleup} \kappa(x):=\big(T(x),\kappa(x)\big)$, is an embedding. Moreover, since  $T$ is an embedding with closed image, Lemma \ref{l:closed-embedding} guarantees that the embedding  $T{\bigtriangleup} \kappa$ also has closed image.

Since  $T{\bigtriangleup} \kappa\big(\kappa X\big)$ is closed in $P\times \beta X$, $X$ is dense in $Y$, and the mappings $T{\bigtriangleup}\kappa$ and $\hat i{\bigtriangleup}R$ coincide on $X$, we obtain $\hat i{\bigtriangleup}R(Y) \subseteq \big(T{\bigtriangleup}\kappa\big)\big(\kappa X\big)$. Therefore there is some $y\in \kappa X$ such that $\hat i{\bigtriangleup}R(z) = \big(T{\bigtriangleup}\kappa\big)(y)$. In particular, $R(z)=\kappa(y)$. Since $R$ and $\kappa$ are bijective, it follows that $z=y\in \kappa X$.
\smallskip

(ii)$\Ra$(iii) Since $\upsilon X$ is a $\kappa$-space by (iii) of Proposition \ref{p:kk-space-permanent}, the inclusion $z\in \upsilon X$ follows from the definition of $\kappa X$. To show that $\delta_z$ is $k$-continuous on $C_p(X)$, let $K$ be a compact subset of $C_p(X)$. Define a mapping $\psi: X\to C_p(K)$ by $\psi(x)(f):=f(x)$ for $f\in K$. Then $\psi$ is continuous, and hence, by (ii), $\psi$ is extended to a continuous mapping $\hat\psi: X\cup\{z\} \to C_p(K)$. Observe that $\delta_z{\restriction}_K =\hat\psi(z)$. Indeed, if $\{x_i\}_{i\in I}\subseteq X$ is a net converging to $z$ and $f\in K$, then
\[
\delta_z(f)=\bar f(z)=\lim_i f(x_i) = \lim_i \psi(x_i)(f)=\hat\psi (z)(f)
\]
and hence $\delta_z{\restriction}_K =\hat\psi(z)$. Thus the mapping $\delta_z{\restriction}_K$ is continuous.
\smallskip

(iii)$\Ra$(ii) Let $K$ be a compact space, and let $\phi:X\to C_p(K)$ be a continuous mapping.  Consider the mapping $\psi:K\to C_p(X)$ defined by $\psi(s)(x):=\phi(x)(s)$, where $s\in K$ and $x\in X$. Then $\psi$ is continuous as the restriction mapping  onto $K$ of the adjoint map $C_p(C_p(K))\to C_p(X)$. Then $\psi(K)$ is a compact subset of $C_p(X)$, and hence, by (iii), the function $\delta_z{\restriction}_{\psi(K)}$ is continuous. Therefore the mapping $\hat\phi:  X\cup\{z\} \to C_p(K)$ defined by
\[
\hat\phi(x)(s):=\delta_x(\psi(s))\quad (x\in X\cup\{z\} \mbox{ and } s\in K)
\]
is well-defined because if $x\in X$ and $s\in K$, we have
\[
\hat\phi(x)(s)=\delta_x\big(\psi(s)\big)=\psi(s)(x)=\phi(x)(s).
\]
In particular, $\hat\phi$ extends $\phi$ and hence $\hat\phi{\restriction}_X$ is continuous. It remains to show that $\hat\phi$ is continuous at $z$. To this end, let $\{x_i\}_{i\in I}\subseteq X$ be a net converging to $z$. Then for every $s\in K$, we have (recall that $\overline{\psi(s)}$ denotes the unique extension of $\psi(s)$ onto $\upsilon X$)
\[
\hat\phi(x_i)(s)=\delta_{x_i}\big(\psi(s)\big)=\delta_{x_i}\big(\overline{\psi(s)}\big)=\overline{\psi(s)}(x_i)\to \overline{\psi(s)}(z)=\delta_{z}\big(\psi(s)\big)=\hat\phi(z)(s).
\]
Thus $\hat\phi$ is continuous also at $z$, and hence $\hat\phi$ is a desired extension of $\phi$.
\smallskip

The implications (iii)$\Ra$(iv)$\Ra$(v) are clear.
\smallskip

(v)$\Ra$(iii) Let $K$ be a compact subset of $C_p(X)$. It follows from Proposition \ref{p:vX-bounded-compact} that $\delta_z$ is bounded on $K$. Choose $n\in\w$ such that
\[
|\delta_z(f)|=|\bar f(z)|\leq n \quad \mbox{ for every } \; f\in K.
\]
Then for every $g\in K$, we obtain
\begin{equation} \label{equ:description-point-kk-1}
\delta_z(g)=\bar g(z)=(Q_n \circ \bar g)(z)=\overline{Q_n \circ g}(z)=\delta_z(Q_n \circ g).
\end{equation}
To prove that $\delta_z$ is continuous on $K$, we observe first that the continuity of $Q_n$ implies that  $Q_n(K)$ is a uniformly bounded compact subset of $C_p(X)$. By assumption $\delta_z$ is continuous on $Q_n(K)$. Since, by (\ref{equ:description-point-kk-1}), $\delta_z{\restriction}_K=\delta_z\circ Q_n{\restriction}_K$ it follows that $\delta_z{\restriction}_K$ is continuous.\qed
\end{proof}

\begin{proposition} \label{p:delta-z-bounded-compact}
Let $X$ be a Tychonoff space, and let $z\in\beta X$. If  $\delta_z$ is bounded on each convergent sequence in $C_p^b(X)$, then $z\in\upsilon X$.
\end{proposition}

\begin{proof}
Suppose for a contradiction that $z\not\in\upsilon X$. Then, by Theorem 3.11.10 of \cite{Eng}, there is a continuous function $h:\beta X\to [0,1]$ such that $h(z)=0$ and $h(x)>0$ for every $x\in\upsilon X$. Choose a null-sequence  $0<a_{n+1}< a_{n}\leq 1$ such that
the set $A_n:= \big\{ x\in\beta X: h(x)\in [a_{n+1},a_n]\big\} \not=\emptyset$ for every $n\in\w$. For every $n>0$, set $B_n:=  \big\{ x\in\beta X: a_n\leq h(x)\big\}$ and choose a continuous function $g_n:\beta X\to [a_n,\tfrac{1}{a_n}]$ such that
\[
g_n{\restriction}_{A_{n+1}\cup\{z\}} =a_n \;\; \mbox{ and }\;\; g_n{\restriction}_{B_{n}}=\tfrac{1}{a_n}.
\]
Finally, for every $n\in\w$, we define $f_n:= \tfrac{1}{g_n}{\restriction}_X\in C^b(X)$. The choice of $h$ implies that for every $x\in X$, there is $n_x\in\w$ such that $x\in B_n$ for every $n\geq n_x$. Therefore $f_n\to \mathbf{0}$ in $C_p^b(X)$. Whence  $S:=\{f_n\}_{n\in\w}\cup\{\mathbf{0}\}$ is a null-sequence in $C_p^b(X)$. However, $\delta_z(f_n)=\tfrac{1}{a_n} \to\infty$, and hence $\delta_z$ is unbounded on $S$. This contradiction shows that $z\in\upsilon X$.\qed
\end{proof}

Now we are ready to prove Theorem \ref{t:x-vX-kX}, Theorem \ref{t:x-vX-kX-bounded} and Theorem \ref{t:compact-Cp(KX)}.

\begin{proof}[Proof of Theorem \ref{t:x-vX-kX}]
(i) This is Proposition \ref{p:vX-bounded-compact}.

\smallskip

(ii) follows from Proposition  \ref{p:description-point-kk}.\qed
\end{proof}

\begin{proof}[Proof of Theorem \ref{t:x-vX-kX-bounded}]
(i) Let $z\in\beta X$. If $z\in\upsilon X$, then, by (i) of Theorem \ref{t:x-vX-kX},  $\delta_z$ is bounded on all compact subsets of $C_p(X)$ and hence also on all compact subsets of $C_p^b(X)$. Conversely, if $\delta_z$ is bounded on all compact subsets of $C_p^b(X)$, then, by Proposition \ref{p:delta-z-bounded-compact}, $z\in\upsilon X$.
\smallskip

(ii) Let $z\in\kappa X$. Then, by (ii) of  Theorem \ref{t:x-vX-kX},  $\delta_z$ is continuous on all compact subsets of $C_p(X)$ and hence also on all compact subsets of $C_p^b(X)$. Conversely, assume that $z\in\beta X$ is such that $\delta_z$ is continuous on all compact subsets of $C_p^b(X)$. Then, by Proposition \ref{p:delta-z-bounded-compact}, $z\in\upsilon X$.  Thus,  by Proposition \ref{p:description-point-kk},  $z\in \kappa X$.\qed
\end{proof}


\begin{proof}[Proof of Theorem \ref{t:compact-Cp(KX)}]
(i)$\Ra$(ii) is clear.
\smallskip

(ii)$\Ra$(iii) Assume that $C_p^b(Y)$ and $C_p^b(X)$ have the same compact subsets. To show that $Y\subseteq \kappa X$, fix an arbitrary $y\in Y$. Then $\delta_y$ is continuous on $C_p^b(Y)$. Therefore  $\delta_y$ is continuous on all compact subsets of $C_p^b(Y)$ and hence of $C_p^b(X)$. Hence, by Theorem \ref{t:x-vX-kX}, $y\in \kappa X$. Thus $Y\subseteq \kappa X$.
\smallskip

(iii)$\Ra$(i) Assume that $Y\subseteq \kappa X$. Since the topology of $C_p^b(Y)$ is defined by the family $\{\delta_y:y\in Y\}$, $C_p^b(Y)$ and $C_p^b(X)$ have the same compact subsets by Proposition \ref{p:description-point-kk}.

It is useful to give a totally different proof of the sufficiency. Since the restriction map $I:C_p^b(Y)\to C_p^b(X)$ is continuous, we always have the inclusion $\KK\big(C_p^b(Y)\big)\subseteq \KK\big(C_p^b(X)\big)$. To prove the inverse inclusion, let $K$ be a compact subset of $C_p^b(X)$. Since $Y\subseteq \kappa X$, Theorem \ref{t:x-vX-kX} implies that each $\delta_y$ is continuous and hence bounded on $K$. Therefore $K$ is a precompact subset of $C_p^b(Y)$. It remains to prove that $K$ is complete in $C_p^b(Y)$. To this end,  let $\{f_\alpha\}_{\alpha\in\AAA}$ be a Cauchy net in $K$ with respect to $C_p^b(Y)$. As $I$ is continuous,  $\{f_\alpha\}_{\alpha\in\AAA}$ is a Cauchy net in $K$ with respect to $C_p^b(X)$. Since $K$ is compact in $C_p^b(X)$, there is $f\in K$ such that $f_\alpha \to f$ in $C_p^b(X)$. Since for every $y\in Y$, $\delta_y$ is continuous on $K$, we have $f_\alpha(y)\to f(y)$. But this means that $f_\alpha \to f\in K$ also in $C_p^b(Y)$. Thus $K$ is complete and hence compact in $C_p^b(Y)$.\qed
\end{proof}

The next example shows that the class of $\kappa$-spaces strictly contains the class of Dieudonn\'{e} complete spaces.

\begin{example} \label{exa:kappa-non-Dieudonne}
There are Fr\'{e}chet--Urysohn  $\kappa$-spaces which are not Dieudonn\'{e} complete.
\end{example}

\begin{proof}
Let $Y=C_p(\w_1+1)$.
From Example \ref{exa:Cp-kappa-non-realcompact} it follows that $Y$ is a Fr\'{e}chet--Urysohn $\kappa$-space, which is not a realcompact space. It remains to note (see Exercise 458 of \cite{Tkachuk-1}) that $C_p$-spaces are Dieudonn\'{e} complete if, and only if, they are realcompact. \qed
\end{proof}

The next example shows that the class of $\kappa$-spaces is strictly contained in the class of all $\mu$-spaces, it also shows that  Theorem \ref{t:compact-Cp(KX)} strongly generalizes Proposition 2.9 of Haydon \cite{Haydon-2}.

\begin{example} \label{exa:mu-non-kappa}
There are $\mu$-spaces which are not $\kappa$-spaces.
\end{example}

\begin{proof}
Let $X$ be the space described in Example \ref{exa:Cp-mu-non-kappa}, and let $Y=C_p(X)$. It follows from Example \ref{exa:Cp-mu-non-kappa} that $Y$ is a $\mu$-space, which is not a $\kappa$-space. \qed
\end{proof}

\bibliographystyle{amsplain}

\end{document}